\documentclass[reqno, 12pt]{amsart}
\usepackage{amssymb,amsmath,amsfonts,amsthm,comment,mathrsfs,times,graphicx}
\usepackage{color}
\usepackage[english]{babel}
\usepackage[all,cmtip]{xy}
\usepackage{lmodern}
\usepackage{enumitem}
\usepackage{geometry}
\usepackage{booktabs}
\newcounter{dummy} \numberwithin{dummy}{section}
\newtheorem{theo}[dummy]{Theorem }
 
 \newtheorem{lem}[dummy]{Lemma}
 \newtheorem{pr}[dummy]{Proposition}
 
\newtheorem{ef}[dummy]{Definition}
\newtheorem{exs}[dummy]{Examples}
\newtheorem{ex}[dummy]{Example}
\newtheorem{rem}[dummy]{Remark}
\newtheorem{pr-ef}[dummy]{Proposition-Definition}

\title[$p$-rational fields and the structure of some modules]{$p$-rational fields and the structure of some modules }
\author[Abdelaziz EL Habibi,M'hammed Ziane]{Abdelaziz EL Habibi $^{(1)}$, M'hammed Ziane $^{(2)}$}
\address{$^{(1)}$ ACSA Laboratory,
Department of Mathematics, Faculty of Sciences, Mohammed First
University, Oujda, Morocco.}
\email{\textcolor[rgb]{0.00,0.00,1.00}{a.elhabibi@ump.ac.ma}}
\address{$^{(2)}$ ACSA Laboratory,
Department of Mathematics, Faculty of Sciences, Mohammed First
University, Oujda, Morocco.}
 \email{\textcolor[rgb]{0.00,0.00,1.00}{ziane12001@yahoo.fr}}
 \keywords{ Iwasawa theory, Class field theory, Cohomology of number fields, $p$-adic analytic structures.}
 \subjclass[2010]{11R23, 11R37, 11S25}
\begin{document}
\maketitle
\renewcommand{\abstractname}{Abstract}
\begin{abstract}
Assume that the field $K$ is $p$-rational. We study the freeness of
the $\Lambda(G_{\infty,S})$-module
$\mathcal{X}=\mathcal{H}^{ab}=\mathrm{\mathrm{G}al}(K_{S\cup
S_p}/K_{\infty,S})^{ab}$. For numerical evidence to our result we
consider the case of fields of the form
$\mathbb{Q}(\sqrt{pq},\sqrt{-d})$.
\end{abstract}

\section{\textbf{Introduction}}
The cohomology theory gives subtle and deep lying arithmetic laws if
we study Galois groups with restricted ramification. In this paper
we consider infinite fields with restricted ramification and try to
study closely the associated Galois groups. Let $K$ be a field which
is of finite or infinite degree over the rationales. If $\Sigma$ is
a set of places of $K$, then we denote by $K_{\Sigma}$ the maximal
pro-$p$-extension of $K$ which is unramified outside $\Sigma$ and
let $G_{\Sigma}(K)$ be its Galois group over $K$. If $K$ is a number
field and $p$ an odd prime number such that $\Sigma$ contains the
set of places of $K$ dividing $p$, then class field theory shows
that the group $G_{\Sigma}(K)^{ab}$ is isomorphic to
$\mathbb{Z}_{p}^{\rho}\times\mathrm{T}_{\Sigma}$, where $\rho$ is
the $\mathbb{Z}_p$-rank and $\mathrm{T}_\Sigma$ is the
$\mathbb{Z}_p$-torsion. Later we will study and consider examples
for which $\rho=1+r_2$ and $\mathrm{T}_{S_p}=0$, such fields are
called $p$-rational \cite{Movahhedi-Nguyen}. Moreover, it is known
that in this case $G_{\Sigma}(K)$ is of cohomological dimension at
most two. We refer the reader to \cite{Neukirch-Schmidt-Wingberg}
for a complete study of Galois groups with
restricted ramification.\\
In the following we consider extensions with restricted ramification
of infinite fields such as the cyclotomic $\mathbb{Z}_p$ of a number
field. More precisely, let $K$ be a number field and $p$ an odd
prime number. Assume that $S$ is a finite set of non-archimedean
places of $K$ and consider the extension $K_{\infty,S}$ where
$K_{\infty}$ is the cyclotomic $\mathbb{Z}_p$-extension of $K$. For
$S$ satisfying $S\cap S_p=\emptyset$ and $N(v)\equiv1\pmod{p}$ for
every places $v\in S$, the group $G_S(K_{\infty})$ has been
considered by several authors (\cite{Salle1}, \cite{Mizusawa-Ozaki},
\cite{Itoh},...), its structure is used to study the Galois group
$G_{\infty,S}:=\mathrm{Gal}(K_{\infty,S}/K)$.\\
In the present paper, the Galois groups we study are considered as
modules over the complete Iwasawa algebra
$\Lambda(G_{\infty,S})=\mathbb{Z}_p[[G_{\infty,S}]]$ define by:
$$\Lambda(G_{\infty,S})=\mathbb{Z}_p[[G_{\infty,S}]]:=\varprojlim_{U}\mathbb{Z}_p[G_{\infty,S}/U],$$
the projective limite with respect to the open normal sub-groups.

Let $K_{S\cup S_p}$ be the maximal pro-$p$-extension of $K$ which is
unramified outside $S\cup S_p$ with Galois group $G_{S\cup
S_p}(K)=\mathrm{Gal}(K_{S\cup S_p}/K)$. We consider the normal
subgroup $\mathcal{H}=\mathrm{Gal}(K_{S\cup S_p}/K_{\infty,S})$ of
$G_{S\cup S_p}(K)$, its abelianizer $\mathcal{X}$ may be endowed
with a structure of $\Lambda(G_{\infty,S})$-module. Using Nakayama's
lemma, we obtain that $\mathcal{X}$ is finitely generated over
$\Lambda(G_{\infty,S})$ (see the proof of Theorem \ref{theo3}). In
the present paper we study the freeness of the
$\Lambda(G_{\infty,S})$-module $\mathcal{X}$. Roughly speaking, we
prove (Theorem \ref{theo3}) that the $\Lambda(G_{\infty,S})$-module
$\mathcal{X}$ is free of rank $r_2$ by using the following theorem
\cite[Proposition 5.6.13]{Neukirch-Schmidt-Wingberg}:

\begin{theo}
\label{theo} Let $\mathcal{G}$ be a pro-$p$-group with finite
presentation and cohomological dimension at most $2$. Let
$\mathcal{H}$ be a normal closed subgroup of $\mathcal{G}$ with
$H^2(\mathcal{H},\mathbb{Q}_p/\mathbb{Z}_p)=0$. Let
$G=\mathcal{G}/\mathcal{H}$ and $\mathcal{X}=\mathcal{H}^{ab}$.
Suppose that $\mathcal{X}$ is finitely generated under
$\Lambda(G)$.\\
Suppose that $cd(G)\leq 2$ and $H^2(\mathcal{G},
\mathbb{Q}_p/\mathbb{Z}_p)=0$. If the natural morphism
$Tor_{\mathbb{Z}_p}\mathcal{G}^{ab}\longrightarrow G^{ab}$ is
injective, then $\Lambda(G)$-module $\mathcal{X}$ is free.
\end{theo}
We are consedering fields $K$ which are $p$-rational in the sense of
\cite{Movahhedi-Nguyen}, \cite{Gras-Jaulent}, \cite{Gras1} and
\cite{Greenberg}. In particular, a number field $K$ is said to be
$p$-rational if the Galois group $G_{S_p}(K)$ is a free
pro-$p$-group on $r_2+1$ generators. The assumption that $K$ is
$p$-rational with the fact that the number $s$ of places $v$ of
$K_\infty$ above places in $S$ is one, gives that $G_{\infty, S}$ is
a $p$-adic analytic group without $p$-torsion.\\
For the statement of the results, we need some standard notations.
Let $\mathcal{U}_v$ be the $p$-adic compactification of the group of
units of the ring of integers of $K_v$ and denote
$\mathcal{U}_p=\prod_{v\in S_p} \mathcal{U}_v$ there product for the
primes above $p$. If $T$ is a finite set of places of $K$, let
$$\iota_T: \mathcal{E}\longrightarrow \prod_{v \in T} \mathcal{U}_v,$$ the local embedding with respect to element of $T$ of
$\mathcal{E}=\mathbb{Z}_p\otimes E_K$ the $p$-adic completion of the
group of units of $K$. Let $\mathcal{E}_T$ be the kernel of
$\iota_T$, we denote $\iota_T$ by $\iota_p$ when $T=S_p$.\\
Using the above notations we can state the following result:
\begin{theo}(Theorem \ref{theo3})
Let $K$ be a $p$-rational field not containing the $p^{th}$-roots of
unity and satisfying the following conditions:
\begin{enumerate}
    \item  The group
    $i_{p}(\mathcal{E}_S)$  is a direct summand of
    $\mathcal{U}_{p}$.
    \item $s=1$.
\end{enumerate}
Then the $\Lambda(G_{\infty, S})$-module $\mathcal{X}$ is free of
rank $r_2$.
\end{theo}
Assume that $\mathcal{G}=G_{S\cup S_p}(K)$ and $G=G_{\infty, S}$
which are pro-$p$-groups. It is know that the cohomological
dimension of $G_{S\cup S_p}$ is at most two. Since $\mathcal{H}$ is
a normal subgroup of $\mathcal{G}$ such that
$H^2(\mathcal{H},\mathbb{Q}_p/\mathbb{Z}_p)=0$ (see \cite[theorem
2.2]{Nguyen3}), and $G_{\infty, S}$ is of cohomological dimension at
most two (see the proof of theorem \ref{theo3}), we can apply
theorem \ref{theo} to these groups where condition (1) of theorem
\ref{theo3} gives the injection
$$Tor_{\mathbb{Z}_p}G^{ab}_{S\cup S_p}(K)\longrightarrow G_{\infty, S}^{ab}.$$
We use the reciprocity map in the pro-$p$-version of global class
field theory to control the kernel of the above map (\cite[Theorem 2.3]{Jaulent} and \cite[Proposition 1.2]{Salle2}).\\
One remark that Proposition \ref{proposition7} gives the
equivalence:
\begin{equation*}
\mathcal{X}\; \hbox{is free}\; \Leftrightarrow\;
\mathrm{Tor}_{\mathbb{Z}_p}\Big(\frac{\prod_{v\in
S_p}\tilde{\mathcal{U}}_v}{\imath_p(\mathcal{E}_S)\cap
 \prod_{v\in S_p} \tilde{\mathcal{U}}_v}\Big)=\{1\}.
\end{equation*}
As an application of our result we consider the fields of the form
$\mathbb{Q}(\sqrt{pq},\sqrt{-d})$, which we suppose to be
$p$-rational, where $q$ is an odd prime number and $p>3$. Using
pari-GP and the technique described in example \ref{example2} we
numerical examples for such fields.\\\\
 \textbf{Acknowledgement.}
The first author would like to thank Christian Maire for many
helpful discussions and remakes on a first draft of this paper. I am
also grateful to Aurel Page for improving my program on PariGP and
Bill Allombert for inviting me many times to the Atelier PariGP.
\begin{center}\textbf{Notations}
\end{center}
Fix a prime number $p>2$ and a number field $K$. We use the
following notation:
\begin{itemize}
    \item $\mu_p$ denote the set of all $p$th roots of unity.
    \item $S_p=\{\mathfrak{p} \in Pl_K, \mathfrak{p}/p \}$.
    \item $S$ : finite set of primes of $K$ which is disjoint from
    $S_p$.
    \item $E_K$ denote the unit group of $K$.
    \item $K_{\infty}$ : the cyclotomic
$\mathbb{Z}_p$-extension of $K$ and put
$\Gamma=\mathrm{Gal}(K_{\infty}/K)$. We denote by $K_n$ the $n$th
layer of $K_{\infty}/K$.
   \item $K_{\infty,S}$ : the maximal pro-$p$ extension of $K_{\infty}$ unramified
   outside $S$.
   \item If $\Sigma$ is any finite set of places of $K$, denote by
   $K_{\Sigma}$ the maximal pro-$p$ extension of $K$ unramified
   outside $\Sigma$ and put
   $G_{\Sigma}(K)=\mathrm{Gal}(K_{\Sigma}/K)$. Let denote
   $\mathfrak{X}_{\Sigma}(K):=\mathrm{Gal}(K_{\Sigma}/K)^{ab}$.
   \item If $A$ is a
$\mathbb{Z}_p$-module, denote by $d_pA:=\mathrm{dim}_{\mathbb{F}_p}
A/A^p$ the $p$-rank of A.
  \item If $G$ is a profinite group, $H$ a closed subgroup of $G$
  and $M$ a compact $\mathbb{Z}_p[[H]]$-module, then $\mathrm{Ind}_G^H M :=
  M\hat{\otimes}_{\mathbb{Z}_p[[H]]} \mathbb{Z}_p[[G]]$ denotes the
  compact induction of $M$ from $H$ to $G$.
\end{itemize}
\section{\textbf{Preliminaries}}
The results of this paragraph figure in \cite{Brumer},
\cite{Neukirch-Schmidt-Wingberg}, \cite{Serre}, \cite{Nguyen2} ...\\
 Let $p$ be a prime number. Let $G$ be a pro-$p$-group. Consider the complete Iwasawa
algebra of $G$ over $\mathbb{Z}_p$
$$\Lambda(G)=\mathbb{Z}_p[[G]]:=\varprojlim_{U}\mathbb{Z}_p[G/U],$$
the projective limite with respect to the open normal sub-groups.\\
 The compact $\mathbb{Z}_p$-algebra is a
local ring with maximal ideal $m_{G}$ generated by the augmentation
ideal $I_{G}$ and $p\Lambda(G)$.\\
Let $\mathcal{C}$ be the abelian category of compact
$\Lambda(G)$-modules. In the following all considered modules
$\mathcal{X}$ lies in the category $\mathcal{C}$. Let recall the
crucial Nakayama lemma
\begin{lem}
A $\Lambda(G)$-module $\mathcal{X}\in\mathcal{C}$ is generated by
$r$ elements if and only if the $\mathbb{F}_p$-vector space
$\mathcal{X}/m_G$ has dimension $r$.
\end{lem}
\begin{ef}
If $\mathcal{X}$ is a free and finitely generated module over
$\Lambda(G)$, then its $\Lambda(G)$-rank is the unique integer
$\rho_{\mathcal{X}}$ satisfying
$$\mathcal{X}\simeq\Lambda(G)^{\rho_{\mathcal{X}}}.$$
\end{ef}
\subsection{Structure and cohomological dimension of $\mathcal{X}$}
Let $G$ be a pro-$p$-group. We say that $G$ has cohomological
dimension $cd(G)=n$ if the group $H^{n+1}(G,\mathbb{F}_p)$ vanishes
while $H^n(G,\mathbb{F}_p)$ is not.\\
We have the following useful equivalence for free pro-$p$-groups:
\begin{pr}
A non-trivial pro-$p$-group $G$ is said to be a free pro-$p$-group
if and only if $H^2(G,\mathbb{Q}_p/\mathbb{Z}_p)=0$ and $G^{ab}$ has
no torsion.
\end{pr}
\noindent{\textbf{Proof}.} Consider the exact sequence
$\displaystyle{1 \longrightarrow \mathbb{Z}_p
\longrightarrow\mathbb{Z}_p
\longrightarrow \mathbb{Z}_p/p\mathbb{Z}_p \longrightarrow 1,}$\\
 used
to obtain the homology exact sequence
\begin{center}
\[
  \xymatrix @C=2pc{...H_2(G, \mathbb{Z}_p)\ar@[>][r]^p&H_2(G, \mathbb{Z}_p)\ar@[>][r]&H_2(G, \mathbb{F}_p)\ar@[>][r]&H_1(G,
  \mathbb{Z}_p)\ar@[>][r]^p&H_1(G, \mathbb{Z}_p)...}
\]
\end{center}
and hence
$$H_2(G, \mathbb{Z}_p)/p\hookrightarrow H_2(G, \mathbb{F}_p)\longrightarrow G^{ab}[p]\longrightarrow 1.$$
It follows that $d_pH_2(G, \mathbb{F}_p)=d_pH_2(G,
\mathbb{Z}_p)+d_pG^{ab}[p]$. \hfill $\square$\vskip 7pt Also we have
the following lemma,
\begin{lem}
Let $G$ be a pro-$p$-group such that for any $i\geq0$ the groups
$H_i(G, \mathbb{Z}_p)$ are finitely generated as
$\mathbb{Z}_p$-modues. Then $H_2(G, \mathbb{Z}_p)=0$ if
and only if $rg_{\mathbb{Z}_p}G^{ab}=1-\chi_2(G).$\\
Where $\chi_2(G)$ is the Euler-Poincaré characteristic of $G$ with
coefficients in $\mathbb{F}_p$.
\end{lem}
In particular if the group $G$ is isomorphic to
$\mathbb{Z}_p\rtimes\mathbb{Z}_p$, the group $H_2(G, \mathbb{Z}_p)$
is trivial if and only if the product
$\mathbb{Z}_p\rtimes\mathbb{Z}_p$ is not direct.
\begin{pr}
Let $\mathcal{X}$ is a finitely generated $\Lambda(G)$-module. Then
$\mathcal{X}$ is free if and only if $H_1(G,\mathcal{X})$ is trivial
and $\mathcal{X}_G:=\mathcal{X}/I_G$ is $\mathbb{Z}_p$-free.
\end{pr}
\noindent{\textbf{Proof}.} Let  $0 \longrightarrow N \longrightarrow
\Lambda(G)^r \longrightarrow \mathcal{X} \longrightarrow 1$
 be a minimal presentation of $\mathcal{X}$, which gives the following
sequence
\begin{center}
\[
  \xymatrix @C=2pc{
  \ar@[>][r]&H_1(G, \mathcal{X})\ar@[>][r]&N_G\ar@[>][r]&\mathbb{Z}_p^r\ar@[>][r]&\mathcal{X}_G\ar@[>][r]&1}.
  \]
\end{center}
The fact that $\mathcal{X}$ is a projective $\Lambda(G)$-module
gives the conclusion. \hfill $\square$\vskip 7pt
\begin{lem}\cite[Chapter II ,\S4 ,Exercise 4]{Neukirch-Schmidt-Wingberg}
Let $\mathcal{G}$ be a pro-$p$-group and $\mathcal{H}$ be a normal
closed subgroup of $\mathcal{G}$ with
$H^2(\mathcal{H},\mathbb{Q}_p/\mathbb{Z}_p)=0$. Let
$G:=\mathcal{G}/\mathcal{H}$. then there is an exact sequence in
homology
\[
  \xymatrix @C=2pc{H_3(G, \mathbb{Z}_p)\ar@[>][r]&H_1(G, \mathcal{H}^{ab})\ar@[>][r]&H_2(\mathcal{G},\mathbb{Z}_p)
  \ar@[>][r]&H_2(G,\mathbb{Z}_p)\ar@[>][r]&\mathcal{H}^{ab}_G\ar@[>][r]&\mathcal{G}^{ab}\ar@{->>}[r]&G^{ab}}
  \]
\end{lem}
\noindent{\textbf{Proof.}}
 we use the Hochschild-Serre spectral sequence
$$H^i(G, H^j(\mathcal{H},
\mathbb{Q}_p/\mathbb{Z}_p))\Longrightarrow H^{i+j}(\mathcal{G},
\mathbb{Q}_p/\mathbb{Z}_p)$$ which gives the seven terms exact
sequence
\begin{center}
\[
  \xymatrix @C=3pc{0\ar@[>][r]&E^{1,0}_2\ar@[>][r]&E^1\ar@[>][r]&E^{0,1}_2
  \ar@[>][r]^{d^{0,1}_2}&E^{2,0}_2\ar@[>][r]&E^{2}_1\ar@[>][r]&E^{1,1}_2\ar@{>}[r]^{d^{1,1}_2}&E^{3,0}_2},
  \]
\end{center}
 where $E^{1,1}_3=\ker d^{1,1}_2$. Then :
\begin{center}
\[
  \xymatrix @C=1.5pc{0\ar[r]&H^1(G,\mathbb{Q}_p/\mathbb{Z}_p)\ar[r]& H^1(\mathcal{G},
\mathbb{Q}_p/\mathbb{Z}_p)\ar[r]&H^1(\mathcal{H},\mathbb{Q}_p/\mathbb{Z}_p)^G\ar[r]&
H^2(G,\mathbb{Q}_p/\mathbb{Z}_p)\ar[d]\\
 & & H^3(G,\mathbb{Q}_p/\mathbb{Z}_p) & H^1(G,H^1(\mathcal{H},\mathbb{Q}_p/\mathbb{Z}_p)) \ar[l] & E^2_1 \ar[l] }
\]
\end{center}
By definition $E_{\infty}^{0,2}$ is a subgroup of
$E_{2}^{0,2}=H^2(\mathcal{H},\mathbb{Q}_p/\mathbb{Z}_p)$, which is
trivial by hypotheses, moreover $E^2/E^2_{1}(=E^{0,2}_{\infty})$ is
also trivial. We obtain that
$E_{1}^2=E^2=H^2(\mathcal{G},\mathbb{Q}_p/\mathbb{Z}_p)$ and it
suffices to apply the Pontryagin duality
\begin{center}
\[
  \xymatrix @C=3pc{H_3(G, \mathbb{Z}_p)\ar[r]&
 H_1(G,H_1(\mathcal{H},\mathbb{Z}_p))\ar[r]&H_2(\mathcal{G},
\mathbb{Z}_p)\ar[r]& H_2(G,\mathbb{Z}_p)\ar[d]\\
 0& H_1(G,\mathbb{Z}_p) \ar[l] & H_1(\mathcal{G},
\mathbb{Z}_p) \ar[l] & H_1(\mathcal{H}, \mathbb{Z}_p)_G \ar[l]  }
\]
\end{center}
\hfill $\square$\vskip 7pt \noindent{\textbf{Proof of Theorem
\ref{theo}.}} Since $cd(G)\leq 2$, we have an exact sequence
\begin{center}
\[
  \xymatrix @C=2pc{0
  \ar@[>][r]&H_2(G,\mathbb{Z}_p)\ar@[>][r]^p&H_2(G,\mathbb{Z}_p)\ar@[>][r]&H_2(G,\mathbb{F}_p)},
\]
\end{center}
and from our assumption, Thus $H_2(G,\mathbb{Z}_p)$ is free of
finite rank as a $\mathbb{Z}_p$-module, $H_1(G, \mathcal{X})$ is
trivial. The spectral sequence $H^i(G, H^j(\mathcal{H},
\mathbb{Q}_p/\mathbb{Z}_p))$ $\Longrightarrow$ $H^{i+j}(\mathcal{G},
\mathbb{Q}_p/\mathbb{Z}_p)$ gives the exact sequence
\begin{center}
\[
  \xymatrix @C=2pc{0
  \ar@[>][r]&H_2(G,\mathbb{Z}_p)\ar@[>][r]&\mathcal{X}_G\ar@[>][r]&\mathcal{G}^{ab}\ar@[>][r]&G^{ab}\ar@[>][r]&0}
  \]
\end{center}
(using $H_2(\mathcal{G}, \mathbb{Z}_p)$) if moreover the natural
morphism $Tor_{\mathbb{Z}_p}\mathcal{G}^{ab}\longrightarrow G^{ab}$
is injective, then, $\mathcal{X}_G$ is $\mathbb{Z}_p$-free.
 \hfill $\square$\vskip 7pt
This proposition is an improvement Theorem \ref{theo} by giving the
other implication:
\begin{pr}
\label{propostion1} Under the hypotheses of the theorem \ref{theo},
suppose that $cd(G)\leq2$ and that $H_2(\mathcal{G},
\mathbb{Z}_p)=H_2(G, \mathbb{Z}_p)=0$. Then the $\Lambda(G)$-module
$\mathcal{X}$ is free if and only if the morphism
$Tor_{\mathbb{Z}_p}\mathcal{G}^{ab}\longrightarrow G^{ab}$ is
injective.
\end{pr}
\noindent{\textbf{Proof}.} We have that $H_2(G, \mathbb{Z}_p)=0$,
then $H_1(G,\mathcal{X})$ is trivial, and the spectral sequence
$H^i(G, H^j(\mathcal{H}, \mathbb{Q}_p/\mathbb{Z}_p))$
$\Longrightarrow$ $H^{i+j}(\mathcal{G}, \mathbb{Q}_p/\mathbb{Z}_p))$
gives the exact sequence
$$\xymatrix @C=2pc{0
  \ar@[>][r]&\mathcal{X}_G\ar@[>][r]&\mathcal{G}^{ab}\ar@[>][r]&G^{ab}\ar@[>][r]&1}$$
hence $\mathcal{X}_G$ is $\mathbb{Z}_p$-free if and only if the
morphism $Tor_{\mathbb{Z}_p}\mathcal{G}^{ab}\longrightarrow G^{ab}$
is injective.\hfill $\square$\vskip 7pt
\subsection{On analytic pro-$p$-groups}
The main references for this part are: Lazard \cite{Lazard}, Dixon,
Du Sautoy, Segal, Mann \cite{Dixon-Du Sautoy-Mann-Segal}, ...
\begin{ef}
A topological group $G$ is $p$-adic analytic if $G$ has a structure
of $p$-adic analytic manifold for which the morphism $(G,
G)\longrightarrow G$: $(x, y)\longrightarrow xy^{-1}$ is analytic.
 \end{ef}
Let $G$ be a pro-$p$-group. If $G$ is $p$-adic analytic, then
$\Lambda(G)$ is noetherian (see \cite[V.2.2.4]{Lazard}). If more $G$
is without $p$-torsion, then
$\Lambda(G)$ is without zero divisor (see \cite{Neumann}).\\
 If $\Lambda(G)$ is
Noetherian and without zero-divisors we can form a skew field $Q(G)$
of fractions of $\Lambda(G)$ (see \cite[chapter
9]{Goodearl-Warfield}). This allows us to define the rank of a
$\Lambda(G)$-module:
\begin{ef}
 Suppose that $G$ is $p$-adic analytic group without $p$-torsion, the $\Lambda(G)$-rank of a finitely generated
$\Lambda(G)$-module $\mathcal{X}$, is defined by
 $$rank_{\Lambda(G)}(\mathcal{X}):=dim_{Q(G)}(Q(G)\otimes_{\Lambda(G)}\mathcal{X}).$$
\end{ef}
\begin{theo}(\cite[proposition 1.1 and theorem 1.4]{Nguyen3})
\label{theo1} Let $\mathcal{G}$ be a pro-$p$-group with finite
presentation and cohomological dimension at most $2$. Let
$\mathcal{H}$ be a normal closed subgroup of $\mathcal{G}$ and
$G:=\mathcal{G}/\mathcal{H}$ is a $p$-adic analytic group without
$p$-torsion. Put $\mathcal{X}=\mathcal{H}^{ab}$. Then $\mathcal{X}$
and $H_2(\mathcal{H},\mathbb{Z}_p)$ are finitely generated under
$\Lambda(G)$ and
$$rank_{\Lambda(G)}(\mathcal{X})=-\chi(\mathcal{G})+ \delta_{G, 1} +rank_{\Lambda(G)}(H_2(\mathcal{H},\mathbb{Z}_p))
,$$ where $\chi(\mathcal{G})=\sum_{i\geq 0}dim_{\mathbb{F}_p}H_i(G,
\mathbb{F}_p)$ is the Euler-Poincaré characteristic of G and
$\delta_{G, 1}=1$ if $G$ is trivial and zero otherwise.
\end{theo}
\begin{rem}
A more general result can be found in the paper \cite[theorem
1.1]{Howson} of Howson.
\end{rem}
\subsection{On the torsion of some Iwasawa modules}
Let $\Sigma$ be a finite set of places of $K$ containing $S_p$. Let
$K_{\Sigma}$ be the maximal $\Sigma$-ramified pro-$p$-extension of
$K$. Let denote
$\mathfrak{X}_{\Sigma}(K):=\mathrm{Gal}(K_{\Sigma}/K)^{ab}$.\\
We denote $\Gamma$ the Galois group of the cyclotomic
$\mathbb{Z}_p$-extension of $K$.\\
By definition we have the following exact sequence :
\begin{center}
\[
  \xymatrix @C=2pc{0
  \ar@[>][r]&Tor_{\Lambda(\Gamma)}\mathfrak{X}_{\Sigma}(K_{\infty})\ar@[>][r]&\mathfrak{X}_{\Sigma}(K_{\infty})
  \ar@[>][r]&fr_{\Lambda(\Gamma)}\mathfrak{X}_{\Sigma}(K_{\infty})\ar@[>][r]&0}
\]
\end{center}
\begin{theo}(\cite[section 2]{Iwasawa})
Suppose that $K_\infty$ satisfy the weak Leopoldt conjecture then we
have :
\begin{enumerate}
    \item $rang_{\Lambda(\Gamma)}\mathfrak{X}_{\Sigma}(K_{\infty})= r_2$
    \item $\mathfrak{X}_{\Sigma}(K_{\infty})$ has no non trivial finite
    sub-module.
\end{enumerate}
\end{theo}
We are going to study the relation between
$\mathrm{Tor}_{\Lambda(\Gamma)}(\mathfrak{X}_{\Sigma}(K_{\infty}))$
which is the projective limit of
$\mathrm{Tor}_{\mathbb{Z}_p}(\mathfrak{X}_{\Sigma}(K_{n}))$.\\
The next theorem is necessary to prove some results in the
following.
\begin{theo} (\cite[Proposition 3.1]{Nguyen1} or \cite[Proposition 3.1]{Nguyen2} )
\label{theo1} Let $K_\infty/K$ be the cyclotomic
$\mathbb{Z}_p$-extension of $K$. Suppose that the Leopoldt
conjecture holds for the fields $K_n$ and the prime $p$ for any
$n\geq 0$. Then
$$\mathrm{Tor}_{\Lambda(\Gamma)}(\mathfrak{X}_{\Sigma}(K_{\infty}))\cong \varprojlim_n \mathrm{Tor}_{\mathbb{Z}_p}\mathfrak{X}_{\Sigma}(K_n).$$
\end{theo}
\noindent{\textbf{Proof}.}
 Using the structure of
$\Lambda(\Gamma)$-modules \cite[theorem
5.1.10]{Neukirch-Schmidt-Wingberg}, we have the following exact
sequences :
\begin{center}
\[
  \xymatrix @C=2pc{0
  \ar@[>][r]&\mathrm{Tor}_{\Lambda(\Gamma)}\mathfrak{X}_{\Sigma}(K_{\infty})
  \ar@[>][r]&\mathfrak{X}_{\Sigma}(K_{\infty})\ar@[>][r]&\mathrm{fr}_{\Lambda(\Gamma)}\mathfrak{X}_{\Sigma}(K_{\infty})\ar@[>][r]&0}
\]
\end{center} and
\begin{center}
\[
  \xymatrix @C=2pc{0
  \ar@[>][r]&\mathrm{fr}_{\Lambda(\Gamma)}(\mathfrak{X}_{\Sigma}(K_{\infty}))\ar@[>][r]&\Lambda(\Gamma)^{r_2}\ar@[>][r]&F\ar@[>][r]&0}
\]
\end{center}
where $F$ is a finite $\Lambda(\Gamma)$-module.\\
By Snake lemma we obtain the following:
\begin{center}
\[
  \xymatrix @C=2pc{0
  \ar@[>][r]&(\mathrm{Tor}_{\Lambda(\Gamma)}\mathfrak{X}_{\Sigma}(K_{\infty}))_{\Gamma_n}\ar@[>][r]&(\mathfrak{X}_{\Sigma}(K_{\infty}))_{\Gamma_n}
  \ar@[>][r]&(\mathrm{fr}_{\Lambda(\Gamma)}\mathfrak{X}_{\Sigma}(K_{\infty}))_{\Gamma_n}\ar@[>][r]&0},
\]
\end{center} and
\begin{center}
\[
  \xymatrix @C=2pc{0
  \ar@[>][r]&F^{\Gamma_n}\ar@[>][r]&(\mathrm{fr}_{\Lambda(\Gamma)}\mathfrak{X}_{\Sigma}(K_{\infty}))_{\Gamma_n}
  \ar@[>][r]&(\Lambda(\Gamma)^{r_2})_{\Gamma_n}\ar@[>][r]&F_{\Gamma_n}\ar@[>][r]&0},
\]
\end{center}
it is now clear that Leopoldt's conjecture holds for all $K_n$,
$n\geq 0$  if and only if divisor of
$\mathfrak{X}_{\Sigma}(K_{\infty})$ is disjoint from all
$\omega_n=\gamma^{p^n}-1$, $n\geq 0$, where $\gamma$ is topological
generator of $\Gamma$(\cite[\S10]{Iwasawa}).\\

A simple calculation of $\mathbb{Z}_p$-rank (we recall that the
fields $K_n$ satisfy the Leopoltd conjecture) and we have that
$$\Gamma_n\cong
\mathfrak{X}_{\Sigma}(K_n)/(\mathfrak{X}_{\Sigma}(K_{\infty}))_{\Gamma_n},$$
which gives the exact sequence
\begin{center}
\[
  \xymatrix @C=2pc{0
  \ar@[>][r]&(\mathrm{Tor}_{\Lambda(\Gamma)}\mathfrak{X}_{\Sigma}(K_{\infty}))_{\Gamma_n}
  \ar@[>][r]&\mathrm{Tor}_{\mathbb{Z}_p}((\mathfrak{X}_{\Sigma}(K_{\infty}))_{\Gamma_n})\ar@[>][r]&F^{\Gamma_n}\ar@[>][r]&0}
\]
\end{center} furthermore since $\Gamma_n\cong
\mathfrak{X}_{\Sigma}(K_n)/(\mathfrak{X}_{\Sigma}(K_{\infty}))_{\Gamma_n}$
we have that
$\mathrm{Tor}_{\mathbb{Z}_p}((\mathfrak{X}_{\Sigma}(K_{\infty}))_{\Gamma_n})\cong
\mathrm{Tor}_{\mathbb{Z}_p}((\mathfrak{X}_{\Sigma}(K_n))$. Passing
to the projective limit gives the required result. \hfill
$\square$\vskip 7pt The next theorem is necessary to prove some
results in the following.
\begin{theo}\cite[Theorem
11.3.5]{Neukirch-Schmidt-Wingberg} \label{theo2} Assume that the
weak leopoldt conjecture holds for the $\mathbb{Z}_p$-extension
$K_{\infty}/K$ and let $\Sigma \supset S_p$ be finite. Then there
exists a canonical exact sequence of $\Lambda(\Gamma)$-modules.
\begin{center}
\[
  \xymatrix @C=2pc{0 \ar@[>][r]&\bigoplus_{v\in \Sigma\setminus S_p}
  \mathrm{Ind}^{\Gamma_{v}}_{\Gamma}(\mathrm{T}(K_{v}(p)/K_{v})_{G_{(K_{\infty})_{v}}})\ar@[>][r]&
  \mathfrak{X}_{\Sigma}(K_{\infty})\ar@[>][r]&\mathfrak{X}_{S_p}(K_{\infty})\ar@[>][r]&0}.
\]
\end{center}
In particular, there is an exact sequence of
$\Lambda(\Gamma)$-torsion modules
\begin{center}
\[
  \xymatrix @C=1.5pc{0 \ar@[>][r]&\bigoplus_{v\in \Sigma\setminus S_p}
  \mathrm{Ind}^{\Gamma_{v}}_{\Gamma}(\mathrm{T}(K_{v}(p)/K_{v})_{G_{(K_{\infty})_{v}}})\ar@[>][r]&
  \mathrm{Tor}_{\Lambda(\Gamma)}(\mathfrak{X}_{\Sigma}(K_{\infty}))\ar@[>][r]&\mathrm{Tor}_{\Lambda(\Gamma)}(\mathfrak{X}_{S_p}(K_{\infty}))\ar@[>][r]&0},
\]
\end{center}
where $K_{v}(p)$ is the maximal pro-$p$-extension of $K_v$,
$\mathrm{T}(K_{v}(p)/K_{v})$ is the inertia subgroup, $\Gamma_v$ is
the decomposition subgroup of $\Gamma$ at $v$ and
$G_{(K_{\infty})_{v}}$ is the absolute Galois group of
$(K_{\infty})_{v}$.
\end{theo}
\subsection{\textbf{$p$-rational fields}}
Number fields $K$ such that
$H^{2}(G_{S_p}(K),\mathbb{Z}/p\mathbb{Z})=0$  are called
$p$-rational \cite{Movahhedi-Nguyen},\cite{Movahhedi-Nguyen},
\cite{Gras-Jaulent} \cite{Gras1},... In particular, a number field
$K$ is $p$-rational precisely when the Galois group $G_{S_p}(F) $ of
the maximal pro-$p$-extension of $F$ which is unramified outside $p$
is pro-$p$-free (with rank $1+r_2$, $r_2$ being the number of
complex primes of $K$). They are first introduced in
\cite{Movahhedi-Nguyen} to construct non-abelian extensions of
$\mathbb{Q}$ satisfying  the Leopoldt conjecture. Recently,
R.Greenberg  \cite{Greenberg} used $p$-rational number fields for
the construction (in a non geometric manner) of $p$-adic
representations with open image in $\mathrm{GL}_{n}(\mathbb{Z}_p)$,
$n\geq3$, of the absolute Galois group $G_{\mathbb{Q}}$.
\begin{pr}\cite{Movahhedi-Nguyen}
The number field $K$ is said to be $p$-rational if the following
equivalent conditions are satisfied:
\begin{enumerate}
\item $K$ satisfies Leopoldt's conjecture and $G_{S_p}^{ab}(K)$ is torsion-free as a
$\mathbb{Z}_p$-module.
\item
\begin{itemize}
        \item $\left\{
  \begin{array}{ll}
    \alpha \in K^{\times} \mid & \hbox{ \begin{tabular}{l}
                                    $\alpha\mathcal{O}_K=\mathfrak{a}^p$ for some fractional ideal $\mathfrak{a}$\\
                                    and $\alpha\in (K^{\times}_\mathfrak{p})^p$ for all $\mathfrak{p}\in S_p$ \\
                                 \end{tabular}
 }
  \end{array}
\right\}=(K^\times)^p$
        \item  and $\delta(K)=\sum_{v\in S_p}\delta(K_v)$.
\end{itemize}
Where for a field $F$, $\delta(F)$ is $1$ if $\mu_p \subset F$ and
$0$ otherwise.
\end{enumerate}
\end{pr}
In the case where $K$ is totally real, it is $p$-rational if and
only if $G_{S_p}^{ab}(K)\cong\mathbb{Z}_p \cong G_{S_p}(K)$.

\begin{exs}If $p=3$ we additionally assume that it is unramified in the first two statements :
\begin{enumerate}
    \item If $K$ is an imaginary quadratic field such
    that $p\nmid h_{K}$, then $K$ is $p$-rational.
    \item If $K$ is a real quadratic field, then it is $p$-rational
    if and only if $p\nmid h_K$ and the fundamental unit is not a
    $p$-power in $K_v$ for all $v|p$.
    \item If the prime $p\geq3$ is regular then the field $\mathbb{Q}(\mu_p)$ is
    $p$-rational.
\end{enumerate}
\end{exs}
Let $\Delta$ be an abelian group and let $\widehat{\Delta}$ denote
the set of the $p$-adic irreducible characters of $\Delta$. Let
$\mathcal{O}$ be the ring of integers of a finite extension of
$\mathbb{Q}_{p}$, which contains the values of all $\chi\in
\widehat{\Delta}$. For any $\mathbb{Z}_{p}[\Delta]$-module $M$ we
define  the $\chi$-quotient $M_{\chi}$ of $M$ by
\begin{equation*}
M_{\chi}=M\otimes_{\mathbb{Z}_{p}[\Delta]}\mathcal{O}(\chi)\cong(M\otimes_{\mathbb{Z}_{p}}\mathcal{O}(\chi^{-1}))_{\Delta}
\end{equation*}
where $\mathcal{O}(\chi)$ denote the ring $\mathcal{O}$ on which
$\Delta$ acts via $\chi$.
\begin{pr}\label{Proposition 6}
Let $M$ be a finite $\mathbb{Z}_{p}[\Delta]$-module. If for all
$p$-adic irreducible characters $\chi$, $M_{\chi}=0$ then $M=0$.
\end{pr}
We have the following interesting proposition which is easy to prove
in the semi-simple case.
\begin{pr}
\label{propostion2} Let $p$ be an odd prime number and $K$ a finite
abelian extension of $\mathbb{Q}$ such that $p\nmid [K:\mathbb{Q}]$,
then the field $K$ is $p$-rational if and only if every cyclic
extension of $\mathbb{Q}$ contained in $K$ is $p$-rational.
\end{pr}
\begin{proof}
It suffices to show the second implication, i.e., if every cyclic
sub-extension is $p$-rational then $K$ is $p$-rational.\vskip 6pt
Let $\Delta$ be the Galois group of $K$ over $\mathbb{Q}$ and denote
$\widehat{\Delta}$ the corresponding group of irreducible character.
Let $\chi$ be such an irreducible character and let $K_{\chi}$ be
its fixed field, if $\Delta_{\chi}=Gal(K/K_{\chi})$ and using the
definition of the $\chi$-quotient, we obtain
\begin{eqnarray*}
  H^{2}(G_{S_{p}}(K),\mathbb{Z}/p\mathbb{Z})_{\chi} &\cong&  (H^{2}(G_{S_{p}}(K),\mathbb{Z}/p\mathbb{Z})\otimes\mathcal{O}(\chi^{-1}))_{\Delta}  \\
   &=& ((H^{2}(G_{S_{p}}(K),\mathbb{Z}/p\mathbb{Z})\otimes\mathcal{O}(\chi^{-1}))_{\Delta_{\chi}})_{\Delta/\Delta_{\chi}}\\
   &\cong&
   (H^{2}(G_{S_{p}}(K),\mathbb{Z}/p\mathbb{Z})_{\Delta_{\chi}}\otimes\mathcal{O}(\chi^{-1}))_{\Delta/\Delta_{\chi}}\\
   &\cong&
   (H^{2}(G_{S_{p}}(K_{\chi}),\mathbb{Z}/p\mathbb{Z})\otimes\mathcal{O}(\chi^{-1}))_{\Delta/\Delta_{\chi}}.
\end{eqnarray*}
Since $K_{\chi}$ is $p$-rational,
$H^{2}(G_{S_{p}}(K_{\chi}),\mathbb{Z}/p\mathbb{Z})=0$. It follows
that
\begin{equation*}
 H^{2}(G_{S_{p}}(K),\mathbb{Z}/p\mathbb{Z})_{\chi}=0.
 \end{equation*}
  Proposition \ref{Proposition 6} shows that
\begin{equation*}
H^{2}(G_{S_{p}}(K),\mathbb{Z}/p\mathbb{Z})=0.
\end{equation*}
This gives the desired implication since the set of cyclic
extensions is one-to-one with the set of irreducible characters.
\end{proof}
For example if $p$ is an odd prime, then
$\mathbb{Q}(\sqrt{d_1},\sqrt{d_2})$ is $p$-rational if and only if
the fields $\mathbb{Q}(\sqrt{d_1})$, $\mathbb{Q}(\sqrt{d_2})$ and
$\mathbb{Q}(\sqrt{d_1 d_2})$ are $p$-rationals.
\section{\textbf{On the structure of the $\Lambda(G_{\infty, S})$-module $\mathcal{X}$}}
Let $K$ be a number field and $p$ an odd prime number. Let $S$ be a
finite set of primes of $K$ which is disjoint from $S_p$ and let
$\Sigma=S\cup S_p$. Recall that $\mathcal{X}$ is the abelianizer of
the group $\mathcal{H}$, where
$\mathcal{H}=\mathrm{Gal}(K_{\Sigma}/K_{\infty,S})$.
\subsection{The structure of the $\Lambda(G_{\infty, S})$-module $\mathcal{X}$}
We are going to use local $p$-adic class field theory in the
following, for this we introduce some notations.\\
Let $v$ be a place of $K$ and let $\mathcal{U}_v$ be the completion
of the $p$-adic units of $\mathcal{O}_v$ of the field $K_v$ defined
as the projective limit
$\lim_n\mathcal{O}^{\times}_v/\mathcal{O}_{v}^{\times p^n}$. The
locally cyclotomic units $\widetilde{\mathcal{U}}_v$ of $K_v$ are
the units of $K_v$ which are norms in the cyclotomic
$\mathbb{Z}_p$-extension $K_{v,\infty}$ of $K_v$.\\
 In particular if
$v\nmid p$, we have $\widetilde{\mathcal{U}}_v=\mathcal{U}_v$. By
local class field theory the group $\widetilde{\mathcal{U}}_v$
correspond to the compositum of the cyclotomic
$\mathbb{Z}_p$-extension of $K_v$ and
the unramified $\mathbb{Z}_p$-extension $K^{un}_v$ of $K_v$.\\
If $T$ is a finite set of places of $K$, set
$\mathcal{U}_{T}=\bigcup_{v\not\in T}\mathcal{U}_v$,
$\widetilde{\mathcal{U}}_{T}=\bigcup_{v\not\in
T}\widetilde{\mathcal{U}}_v$ and $\mathcal{U}_p=\prod_{v\in S_p}
\mathcal{U}_v$. We are interested in the freeness of the
$\Lambda(G_{\infty,S})$-module $\mathcal{X}$.
\begin{theo} \label{theo3}
Let $K$ be a $p$-rational field not containing the $p^{th}$-roots of
unity and satisfying the following conditions:
\begin{enumerate}
    \item  The group
    $i_{p}(\mathcal{E}_S)$  is a direct summand of
    $\mathcal{U}_{p}$.
    \item $s=1$.
\end{enumerate}
Then the $\Lambda(G_{\infty, S})$-module $\mathcal{X}$ is free of
rank $r_2$.
\end{theo}
\noindent{\textbf{Proof.}} By class field theory and thanks to the
weak Leopoldt conjecture (which is true for the cyclotomic
$\mathbb{Z}_p$-extension) one has the following exact sequence
(theorem \ref{theo2})
\begin{center}
\[
  \xymatrix @C=1.5pc{0 \ar@[>][r]&\bigoplus_{v\in \Sigma\setminus S_p}
  \mathrm{Ind}^{\Gamma_{v}}_{\Gamma}(\mathrm{T}(K_{v}(p)/K_{v})_{G_{(K_{\infty})_{v}}})\ar@[>][r]&
  \mathrm{Tor}_{\Lambda(\Gamma)}(\mathfrak{X}_{\Sigma}(K_{\infty}))\ar@[>][r]&\mathrm{Tor}_{\Lambda(\Gamma)}(\mathfrak{X}_{S_p}(K_{\infty}))\ar@[>][r]&0},
\]
\end{center}
since $\mu_p\subset K_{v}$ for $v\in S$ and $v$ is finitely
decomposed in $K_{\infty}/K$,
$\mathrm{T}(K_{v}(p)/K_{v})_{G_{(K_{\infty})_{v}}}$ is isomorphic to
$\mathbb{Z}_p (1)$. Hence we obtain the exact sequence,
\begin{center}
\[
  \xymatrix @C=3pc{0 \ar@[>][r]&\mathcal{W}\ar@[>][r]&\mathrm{Tor}_{\Lambda(\Gamma)}(\mathfrak{X}_{\Sigma}(K_{\infty}))
  \ar@[>][r]&\mathrm{Tor}_{\Lambda(\Gamma)}(\mathfrak{X}_{S_p}(K_{\infty}))\ar@[>][r]&0}
\]
\end{center}
 where $\mathcal{W}$ is the direct sum of $s$ copy of
$\mathbb{Z}_p(1)$.\\
Since $K$ is $p$-rational, hence $K_n$ is $p$-rational, by theorem
\ref{theo1}, we see that
$\mathrm{Tor}_{\Lambda(\Gamma)}(\mathfrak{X}_{S_p}(K_{\infty}))$ is
trivial, we have that
$\mathrm{Tor}_{\Lambda(\Gamma)}(\mathfrak{X}_{\Sigma}(K_{\infty}))\simeq
\mathcal{W}\simeq\mathfrak{X}_{S}(K_{\infty})$, and hence
$\mathfrak{X}_{S}(K_{\infty})\simeq\mathbb{Z}_p$, we can see that
$G_{\infty, S}\simeq\mathbb{Z}_p\rtimes\mathbb{Z}_p$, hence the
cohomology group $H_2(G_{\infty, S},\mathbb{Z}_p)$ is trivial. The
$\mathbb{Z}_p$-freeness of $\mathfrak{X}_S(K_{\infty})$ gives that
the group $G_S(K_{\infty}))$ is also $\mathbb{Z}_p$-free of rank
$1$, and hence
$$cd(G_{\infty, S})\leq cd(G_S(K_{\infty}))+
cd(\Gamma)=2,$$ the cohomological dimension of
$\mathrm{G}_{\infty,S}$ is at most $2$. It is known that
$G_{\Sigma}(K)$
is of cohomolocical dimension at most $2$.\\
Let prove that $\mathcal{X}$ is finitely generated as
$\Lambda(G_{\infty, S})$-module.\\
Since $G_{\infty, S}\simeq\mathbb{Z}_p\rtimes\mathbb{Z}_p$,
$\Lambda(G_{\infty, S})$ is Noetherian, then $G_{\infty, S}$ is of
finite presentation.\\
The Hochschild-Serre spectrale sequence applied to the short exact
sequence
$$1\longrightarrow \mathcal{H} \longrightarrow G_{\Sigma}(K) \longrightarrow G_{\infty, S} \longrightarrow 1$$
shows that
\begin{center}
\[
  \xymatrix @C=2pc{...H^1(G_{\Sigma}(K), \mathbb{F}_p)\ar@[>][r]&H^1(\mathcal{H}, \mathbb{F}_p)^{G_{\infty, S}}\ar@[>][r]&
  H^2(G_{\infty, S}, \mathbb{F}_p)...}
\]
\end{center}
Since $H^1(G_{\Sigma}(K), \mathbb{F}_p)$ and $H^2(G_{\infty, S},
\mathbb{F}_p)$ are finite,
$\mathcal{X}_{G_{\infty,S}}/p=(H^1(\mathcal{H},
\mathbb{F}_p)^{G_{\infty, S}})^{\ast}$ is also finite, where $\ast$
is the Pontrjagin dual. Then by Nakayama's lemma, one has
$\mathcal{X}$ is finitely generated as
$\Lambda(G_{\infty, S})$-module.\\
 The field $K$ satisfy the Leopoldt conjecture, then the groups
$H^2(G_{\Sigma}(K), \mathbb{Q}_p/\mathbb{Z}_p)$ and
$H^2(\mathcal{H}, \mathbb{Q}_p/\mathbb{Z}_p)$ are trivial
(see\cite{Nguyen3}). By the theorem \ref{theo} the
$\Lambda(G_{\infty,S})$-module $\mathcal{X}$ is free if the morphism
of restriction
$\mathrm{Tor}_{\mathbb{Z}_p}G_{\Sigma}(K)^{ab}\longrightarrow
G_{\infty, S}^{ab}$ is injective.\\
We have the following exact sequence
\begin{center}
\[
  \xymatrix @C=4pc{1
  \ar@[>][r]&\frac{\mathcal{K}^{\times}\widetilde{\mathcal{U}_{S}}}{\mathcal{K}^{\times}\mathcal{U}_{\Sigma}}
  \ar@[>][r]&\frac{\mathcal{J}_{\mathcal{K}}}{\mathcal{K}^{\times}\mathcal{U}_{\Sigma}}
  \ar@[>][r]&\frac{\mathcal{J}_{\mathcal{K}}}{\mathcal{K}^{\times}\widetilde{\mathcal{U}_{S}}}\ar@[>][r]&1},
  \]
\end{center}
where $\mathcal{K}^{\times}=\varprojlim_n
K^{\times}/(K^{\times})^{p^n}$ is the $p$-adic compactified of
$K^\times$ and $\mathcal{J}_{\mathcal{K}}$ is the $p$-adic
compactified of the group of idèles of $K$.\\
By the class field $p$-adic correspondence we have the following
exact sequence
\begin{center}
\[
  \xymatrix @C=3pc{1
  \ar@[>][r]&\mathrm{Gal}(K_\Sigma/K_{\infty, S})^{ab}\ar@[>][r]&\mathrm{Gal}(K_\Sigma/K)^{ab}\ar@[>][r]&\mathrm{Gal}(K_{\infty, S}/K)^{ab}\ar@[>][r]&1},
  \]
\end{center}
and note that
$$\frac{\mathcal{K}^{\times}\widetilde{\mathcal{U}_{S}}}{\mathcal{K}^{\times}\mathcal{U}_{\Sigma}}\simeq
   \frac{\prod_{v\in S_p} \tilde{\mathcal{U}}_v}{\imath_p(\mathcal{E}_S)\cap \prod_{v\in S_p} \tilde{\mathcal{U}}_v}.$$
By hypotheses $\imath_p(\mathcal{E}_S)$ is a direct summand of
$\mathcal{U}_p$, and the group $\mathcal{U}_{p}$ is a free
$\mathbb{Z}_p$-module of rank $[K:\mathbb{Q}]$, then we have also
$$\mathrm{Tor}_{\mathbb{Z}_p}\Big(\frac{\prod_{v\in S_p}\tilde{\mathcal{U}}_v}{\imath_p(\mathcal{E}_S)\cap
 \prod_{v\in S_p} \tilde{\mathcal{U}}_v}\Big)=\{1\}.$$
\hfill $\square$\vskip 7pt
\begin{pr}
\label{proposition7} Let $K$ be a $p$-rational field. Suppose that
$s=1$. Then the $\Lambda(G_{\infty, S})$-module $\mathcal{X}$ is
free of rank $r_2$ if and only if,
$$\mathrm{Tor}_{\mathbb{Z}_p}\Big(\frac{\prod_{v\in S_p}\tilde{\mathcal{U}}_v}{\imath_p(\mathcal{E}_S)\cap
 \prod_{v\in S_p} \tilde{\mathcal{U}}_v}\Big)=\{1\}.$$
\end{pr}
\noindent{\textbf{Proof.}}
 Using proposition \ref{propostion1} and the proof of theorem \ref{theo3} one obtains
the desired result. \hfill $\square$\vskip 7pt
\subsection{Examples}\label{example1}
We consider the following field $K=\mathbb{Q}(\sqrt{pq},\sqrt{-d})$
where $p$ and $q$ are two distinct odd primes such that $p>3$ and
$q\equiv-1\pmod{p}$, $d$ be a positive squarefree integer such that
$p\nmid d$ and $q\nmid d$. Suppose that $-d$ is not a quadratic
residue modulo $p$ and $q$. Since $p$ and $q$ are ramified in the
field $K^{+}=\mathbb{Q}(\sqrt{pq})$ there is one prime
$\mathfrak{p}$ above $p$ and one prime $\mathfrak{q}$ above $q$. Let
$\epsilon$ be the fundamental unit of $K^{+}$. Let
$S_p=\{\mathfrak{p}\}$, $S=\{\mathfrak{q}\}$ and denote
$\Sigma=S\cup S_p$, $L_1=\mathbb{Q}(\sqrt{-dpq})$ and
$L_2=\mathbb{Q}(\sqrt{-d})$.
\begin{pr}
\label{propostion4} The field $K$ is $p$-rational if and only if the
following conditions are satisfied :
\begin{itemize}
    \item the class numbers of $K^{+}$, $L_1$ and $L_2$ are prime to $p$,
    \item the fundamental unit $\epsilon$ is not a $p$-power in the
          completion $K_{\mathfrak{p}}^{+}$.
\end{itemize}
\end{pr}
We have the following theorem on the freeness of $\mathcal{X}$.
\begin{theo} \label{theo4}
Suppose that $K$ is a $p$-rational field. Suppose that
$\imath_p(\mathcal{E}_S)$ is a direct summand of
$\mathcal{U}_{\mathfrak{p}}$ and let $s=1$. Then the
$\Lambda(G_{\infty, S})$-module $\mathcal{X}$ is free of rank $2$ :
$$\mathcal{X}\simeq \mathbb{Z}_p[[\mathbb{Z}_p \rtimes \mathbb{Z}_p]]\oplus \mathbb{Z}_p[[\mathbb{Z}_p \rtimes \mathbb{Z}_p]].$$
\end{theo}
\begin{ex}\label{example2}
Using Pari-GP (\cite{PARI}) we verified the following example where
$K=\mathbb{Q}(\sqrt{91},\sqrt{-2})$, $p=7$, $q=13$ and $d=2$. We
denote $K^{+}$ be the maximal real subfield of $K$, i.e.,
$K^{+}=\mathbb{Q}(\sqrt{91})$. We have $S_p=\{\mathfrak{p}\}$,
$S=\{\mathfrak{q}\}$ and
$\Sigma=\{\mathfrak{p},\mathfrak{q}\}$. Let $\mathfrak{p}|p$ and $\mathfrak{q}|q$.\\
Let $A_n$ (resp. $A_{n}^+$) be the $p$-part of the class group of
$K_n$ (resp. $K_{n}^+$) and $A_{n,S}$ (resp. $A_{n,S}^+$) the
$p$-part of the class group of $K_n$ (resp. $K_{n}^+$) of ray $S$.
For a large enough integer $n$, we have the following (see
\cite{Gras1})
$$\frac{\sharp A_{n,S}}{\sharp A_{n,S}^+}=\frac{\sharp A_n}{\sharp A_{n}^+}\prod_{v\in S(K_n)}p^{n\alpha_{S}+o(1)},$$
where $\alpha_{S}=\sum_{v\in S(K^+)}\alpha_v$ and $\alpha_v$ is $0$
or $1$ such that $\alpha_v=1$ if the following is satisfied
\begin{enumerate}
    \item $|\mathcal{U}_{K^{+}_{v}}|=1$,
    \item $|\mathcal{U}_{K_{v}}|\neq1$.
\end{enumerate}
In this case we have that $\alpha_S=1$.\\
The fundamental unit of the field $K^+$ is not a $p$-power in
$K^{+}_{\mathfrak{p}}$ and $p$ is prime to the class number of $K^+$
and the class number of $L$. Then by proposition \ref{propostion4}
the field $K$ is $p$-rational.\\
Since $A_{0}=\{1\}$ and $\alpha_{S}=1$, we have that the invariant
$\lambda_{S}$ of the module $\mathfrak{X}_S(K_{\infty})$ equals $1$,
hence we obtain that $d_p(A_{0,S})=1$, and using Nakayama lemma we
deduce the following isomorphism
$$\mathfrak{X}_{S}(K_{\infty})\simeq\mathbb{Z}_p\simeq \mathcal{W},$$
where $\mathcal{W}$ is a direct sum of $s$ copy of
$\mathbb{Z}_{p}(1)$ and $s$ is the number of places of $K_{\infty}$
above places in $S$ such that $K_{\infty,v}$ contain $\mu_p$, in
this case $s=1$.\\
We have that $K_{\mathfrak{q}}$ contains the $7$-th roots of unity.
It remains to prove that $\imath_7(\mathcal{E}_S)$ is a direct
summand of $\mathcal{U}_{\mathfrak{p}}$. We have
$K=\mathbb{Q}(\theta)$, where $\theta$ satisfy
$\theta^4-178\theta^2+8649=0$. In this case
$E_{K}=<\pm1,\varepsilon>$, with
$\varepsilon=55/62\theta^3-14905/62\theta+1574$. We verify that
$v_{\mathfrak{p}}(\varepsilon^2-1)=v_{\mathfrak{q}}(\varepsilon^2-1)=1$,
we have $\iota_7(\mathcal{E}_S)=\iota_7(<\varepsilon^2>)$ is a
direct summand of $\mathcal{U}_{\mathfrak{p}}$.\\
Then the $\Lambda(G_{\infty, S})$-module $\mathcal{X}$ is free of
rank $2$ :
$$\mathcal{X}\simeq \mathbb{Z}_7[[\mathbb{Z}_7 \rtimes \mathbb{Z}_7]]\oplus \mathbb{Z}_7[[\mathbb{Z}_7 \rtimes \mathbb{Z}_7]].$$
\end{ex}
Using corollary 3.2 of \cite{Gras2} we give an algorithm for testing
the $p$-rationality of a family of number fields of the form
$K=\mathbb{Q}(\sqrt{pq},\sqrt{-2})$, with $p$ and $q$ given as in
the example \ref{example1}. The algorithm gives also the primes
which satisfy the conditions of example \ref{example2}.\\ In the
following table we give numerical evidence for the existence of
fields $K$ which satisfy the conditions of theorem \ref{theo4}.
\begin{center}
\begin{tabular}{|*{2}{c|}}
    \hline
     $p$  & $q$  \\
    \hline
        &\{79,109,239,359,389,439,599,719,829,1039,1319,1429,1439,1879,2239,2269,2309,2399,2549\\
      5 &2719,2749,2789,2879,2909,2999,3079,3109,3229,3359,4079,4349,4519,4639,4679,4759,4919\\
        &5279,5309,5879,6079,6199,6359,6599,6679,6829,6959,7109,7559,7759,7829,8389,8429,8629\\
        &8719,8999,9199,9319,9479,9679,9719,9839,9949\}\\
    \hline
        &\{13,167,181,223,461,503,727,797,853,1021,1063,1231,1399,1511,1567,1637,1693,1847,1973\\
      7 &2029,2141,2351,2477,2687,3037,3527,3541,3709,3821,3863,3877,3919,4157,4423,4493,4549\\
        &4591,4703,5039,5333,5431,5501,5557,5879,6047,6173,6229,6271,6397,6719,6733,7013,7237\\
        &7349,7559,7573,7727,7853,7951,8287,8861,9239,9421,9463,9533,9743\}\\
    \hline
       &\{103,181,311,389,701,727,1039,1117,1637,1663,1871,1949,2053,2287,3119,3821,4133,4159\\
    13 &4679,4783,5303,5407,5693,5927,6343,6551,6863,6967,7487,7591,7669,8111,8293,8423,8839\\
       &9151,9463\}\\
       \hline
    23 &\{367,919,1103,1471,2069,2207,2437,2621,3541,3863,4093,4231,4783,4967,5197,5381,5519\\
       &5749,6301,6991,7589,7727,8647,8693,8831,9199,9613\}\\

     \hline
     29  & \{173,463,2029,2087,2551,4639,6263,6959,9221,9511,9743\}\\
     \hline

     31  &\{61,557,743,991,1301,1487,1549,2293,3037,3533,3719,3967,4463,5021,6199,7253,7687\\
     &8431,8741,9733\}\\
     \hline
     37 &\{887,2663,3847,4957,5623,6733,7103,7621,8287,9397,9767\}\\
     \hline
     47 &\{751,1597,1879,1973,3853,4229,5639,7237,8647,8741\}\\
     \hline
     53 &\{2543,2861,3391,4133,4663,5087,6359,7207\}\\
     \hline
     61 &\{487,853,1951,2927,4391,6709,8783\}\\
     \hline
     71 & \{709,1277,3407,5821,6247,6389,7951,8093\}\\
     \hline
     79 & \{157,631,2053,4423,7109,7583,7741,9479\}\\
     \hline
     101 & \{2423,7069,8887\}\\
     \hline
     103 & \{823,2677,4943,7621,9887\}\\
      \hline
      109 & \{653,5231,8501,8719\}\\
     \hline
     127 & \{3301,5333,9397\}\\
     \hline
    149 &\{ 2383,7151\}\\
    \hline
   151 &\{ 3623,4831,7247,7549\}\\
   \hline
   157 &\{ 941,3767,5023\}\\
   \hline
167 & \{1669,2671,4007,6679,7013\}\\
 \hline
173 &\{ 2767\}\\
 \hline
181 &\{ 1447,5791\}\\
 \hline
191 & \{4583,7639\}\\
 \hline
197 & \{1181,7879\}\\
 \hline
199 & \{397,3581,6367,9551,9949\}\\
 \hline
 223 & \{1783,4013,5351\}\\
 \hline
229 & \{1373,1831\}\\
 \hline
239 &\{2389,3823\}\\
 \hline
263 &\{4733,6311,8941\}\\
 \hline
271 &\{541,4877\}\\
 \hline
277 & \{3877,8863\}\\
 \hline
311 &\{3109\}\\
 \hline
317 &\{1901,7607\}\\
 \hline
349 & \{2791\}\\
\hline
359 &\{5743\}\\
\hline
367 &\{733,8807\}\\
\hline
373 & \{2237,8951\}\\
\hline
431 & \{7757\}\\
\hline
\end{tabular}
\end{center}

\end{document}